\documentclass{article}
\usepackage{amsmath,amsthm,amssymb}

\usepackage[top=3cm,bottom=3cm,left=3cm,right=3cm,marginparwidth=1.75cm]{geometry}

\usepackage[colorlinks=true,allcolors=blue,pdfpagelabels=false]{hyperref}

\newtheorem{theorem}{Theorem}

\newtheorem{lemma}[theorem]{Lemma}

\newtheorem{corollary}[theorem]{Corollary}
\newtheorem{conjecture}{Conjecture}
\newtheorem{question}[conjecture]{Question}

\theoremstyle{definition}

\theoremstyle{remark}
\newtheorem{remark}{Remark}
\newtheorem{example}{Example}

\usepackage{ifthen}
\newcommand{\mythmname}{}
\newtheoremstyle{mytheorem}
	{5pt}
	{5pt}
	{\it}
	{}
	{}
	{{\bf .}}
	{.5em}
	{\mythmname{\ifthenelse{ \equal{#3}{} }{}{\ (\thmnote{#3})}}}
\theoremstyle{mytheorem}
\newtheorem{namedtheorem}{Name}
\newcommand{\renametheorem}[1]
{
	\renewcommand{\mythmname}{{\bf #1}}
}

\newcommand{\G}{\mathcal{G}}
\newcommand{\A}{\mathcal{A}}
\newcommand{\B}{\mathcal{B}}
\newcommand{\IC}{\mathbb{C}}
\newcommand{\IK}{\mathbb{K}}
\newcommand{\IN}{\mathbb{N}}
\newcommand{\IQ}{\mathbb{Q}}

\newcommand{\IZ}{\mathbb{Z}}
\DeclareMathOperator{\Spec}{Spec}

\title{An algorithmic approach to the Polydegree Conjecture for plane polynomial automorphisms}

\author{Drew Lewis \thanks{Department of Mathematics and Statistics, University of South Alabama. \texttt{drewlewis@southalabama.edu}.}
\and Kaitlyn Perry \thanks{Department of Mathematics, Wingate University. \texttt{k.perry@wingate.edu}.}
\and Armin Straub \thanks{Department of Mathematics and Statistics, University of South Alabama. \texttt{straub@southalabama.edu}.}
}

\begin{document}
\maketitle

\begin{abstract}
We study the interaction between two structures on the group of polynomial automorphisms of the affine plane: its structure as an amalgamated free product and as an infinite-dimensional algebraic variety.  We introduce a new conjecture, and show how it implies the Polydegree Conjecture.  As the new conjecture is an ideal membership question, this shows that the Polydegree Conjecture is algorithmically decidable.  We further describe how this approach provides a unified and shorter method of recovering existing results of Edo and Furter.  
\end{abstract}

\section{Introduction}

Let $R$ be a commutative, unital ring, and let $R^{[n]}$ denote the $n$-variable polynomial ring over $R$.  We let $\G(R)$ denote the group of automorphisms of $\Spec R^{[2]}$ over $\Spec R$; in particular, when $R=\IC$, this is the group of polynomial automorphisms of the affine plane.  A fundamental area of inquiry in affine algebraic geometry is to try and understand various structures of this group. 

It is well known that, for any field $\IK$, $\G(\IK)$ has the structure of an amalgamated free product.  Let $\A(\IK)$ denote the affine subgroup, consisting of all degree one automorphisms, and let $\B(\IK)$ denote the (upper) triangular subgroup, consisting of all automorphisms of the form  $\left(aX+P(Y),bY+c\right)$
for some $a,b \in \IK^*$, $c \in \IK$, and $P(Y) \in \IK[Y]$.  The classical Jung-van der Kulk Theorem states that $\G(\IK)$ is the amalgamated free product of $\A(\IK)$ and $\B(\IK)$ over their intersection; thus,  for any $\theta \in \G(\IK)$, we can write $$\theta = \alpha _0 \tau _1 \alpha _1 \cdots \tau _k \alpha _k$$ for some $\alpha _i \in \A(\IK)$ and $\tau _i \in \B(\IK) \setminus \A(\IK)$.  While this factorization is not unique, the length $k$ and the degrees of the triangular automorphisms are; this allows us to define the {\em polydegree} of $\theta$ as the sequence $(\deg \tau _1, \ldots, \deg \tau _k)$.  We let $\G_{(d_1,\ldots,d_k)}(\IK)$ denote the subset of automorphisms with polydegree $(d_1,\ldots,d_k)$; in the case of $\IK=\IC$, we write $\G=\G(\IC)$ and $\G_{(d_1,\ldots,d_k)}=\G_{(d_1,\ldots,d_k)}(\IC)$ for conciseness.  

$\G=\G(\IC)$ also has the structure of an infinite dimensional algebraic variety (ind-variety), as it is a locally closed subset of $(\IC[X,Y])^2$ (see \cite{Shafarevich81} or \cite{BCW} for details).  The problem of trying to understand this structure has seen renewed interest recently (see \cite{Blanc2016,Edo2018,EdoPoloni,Furter15,FurterPoloni}, among many others, as well as \cite{FurterKraft} for a very recent and comprehensive survey).  In this paper, we are interested in the more focused question of how the ind-variety structure interacts with the amalgamated free product structure.  
The general question is the following, where for any $A \subset \G$, $\bar{A}$ denotes the closure of $A$ in the Zariski topology.

\begin{question}
What can be said about the varieties $\overline{\G_{(d_1,\ldots,d_k)}}$?
\end{question}

Or, more specifically:

\begin{question} \label{q2}
What conditions on the sequences $(d_1,\ldots,d_k)$ and $(e_1,\ldots,e_l)$ guarantee $\G_{(d_1,\ldots,d_k)} \subset \overline{\G_{(e_1,\ldots,e_l)}}$?
\end{question}

The first result applicable to this question is due to Friedland and Milnor \cite{FriedlandMilnor}, who (using topological methods) showed that $\G_{(d_1,\ldots,d_k)}$ is a constructible subset of $\G$ of dimension $d_1+\cdots+d_k+6$, which implies that, for distinct degree sequences, $d_1+\cdots+d_k < e_1+\cdots+e_l$ is a necessary condition for Question \ref{q2}. 
Furter \cite{Furter02} showed that we must also have $k \leq l$, and later \cite{Furter09} that, when the polydegrees have the same length (i.e. $k=l$), the containment holds if and only if $d_i \leq e_i$ for each $1 \leq i \leq k$.

In the length two case ($l=2$), this means $k=1$ is the only unsettled situation; in this case, it appears that the topological constraint is sufficient.  In fact, the most optimistic possibility seems to be true; this assertion is known as the Polydegree Conjecture.

\renametheorem{Polydegree Conjecture}
\begin{namedtheorem} Let $d,e >2$.
$$\overline{\G_{(d,e)}} = \coprod _{(d',e')<(d,e)} \G_{(d',e')} \cup \coprod _{f<d+e} \G_{(f)}$$
\end{namedtheorem}

In length three ($l=3$), the situation is much more delicate, and it is known that the corresponding conjecture must be false; see \cite{SFPA1, SFPA3} for some results there.  We conjecture more generally

\begin{conjecture}\label{c3}
Let $k<l$, and let $(d_1,\ldots,d_k)$ and $(e_1,\ldots,e_l)$  be two degree sequences with $d_1+\cdots+d_k < e_1+\cdots+e_l$.  Then $\G_{(d_1,\ldots,d_k)} \subset \overline{\G_{(e_1,\ldots,e_l)}}$.
\end{conjecture}

In this paper, we restrict our attention to the length two case, in which there are two results of note:

\begin{theorem}[Edo \cite{SFPA2}]
Let $d,e \in \IN$.  If $d|e$ or $e|d$, then $\mathcal{G}_{(d+e+1)} \subset \overline{\mathcal{G}_{(d+1,e+1)}}$.
\end{theorem}

\begin{theorem}[Furter \cite{Furter15}]
Let $d,e \in \IN$.  If $d \leq 2$ or $e \leq 2$, then $\G_{(d+e+1)} \subset \overline{G_{(d+1,e+1)}}$.
\end{theorem}

Furter's proof involved showing that the Polydegree Conjecture is equivalent to a conjecture he termed the Rigidity Conjecture, and proving the corresponding case of that conjecture.  We also note that Edo and van den Essen \cite{EdoArno} showed that the Polydegree Conjecture is equivalent to a third conjecture they termed the Strong Factorial Conjecture.


In this paper, we take a direct approach with the hopes that the techniques will prove more generalizable to the length 3 case.  We prove three successively stronger results (Theorems \ref{thm:ICDC}, \ref{thm:partialSpecializationSomeC}, and \ref{thm:partialSpecialization}) which provide means of concluding that Conjecture \ref{c3} holds in certain cases.  We state these in Section \ref{sec:theorems}, and prove them in Section \ref{sec:proof}.  Then, in Section \ref{sec:applications}, we show how our methods provide a unified (and much shorter) method of recovering the existing results of Edo and Furter.  We also prove that the Polydegree Conjecture is algorithmically decidable (for fixed $e$ and $d$), and use a computer to show that it holds for $d< 50$ in the $e=3$ case, and $d < 20$ in the $e=4$ case.  Finally, we affirmatively answer a question of Arzhantsev that arises naturally from \cite{ZAK} regarding the infinite transitivity of certain group actions on \(\IC^2\).

\subsection{A new conjecture}
We first establish some notation needed to state a conjecture, which we will prove (Theorem \ref{thm:ICDC}) implies the Polydegree Conjecture. 
For notational convenience, here and throughout, given a sequence of natural numbers ${\bf a}=(a_1,\ldots,a_e)$, we define
\begin{align*}
|{\bf a}|&=a_1+\cdots+a_e, & {\bf a} \cdot \IN &= a_1+2a_2+3a_3+\cdots+ea_e.
\end{align*}
For each $d,e \in \IN$, define $g_{d,e} \in \IZ[x_1,\ldots,x_e]$ by 
$$g_{d,e} = \frac{1}{d+1} \sum _{{\bf a} \cdot \IN = d} (-1)^{|{\bf a}|} \binom{d+|{\bf a}|}{{\bf a}, d} x_1 ^{a_1} \cdots x_e ^{a_e}, $$
where we used the usual notation for the multinomial coefficient.
These polynomials arise naturally in our proofs (see Lemma \ref{lem:vij}), and also arise out of the approach of Edo and van den Essen \cite{EdoArno}.
For each $i,j \in \IZ$ and $e \in \IN$, define $\alpha _{i,j,e} \in  \IZ[x_1,\ldots,x_e]$ by
\begin{equation}
\alpha _{i,j,e} = \sum _{{\bf a} \cdot \IN = j-i} (-1)^{|{\bf a}|} \binom{|{\bf a}|+j+2}{{\bf a}, j+2} x_1 ^{a_1}\ \cdots x_e ^{a_e}. \label{eq:alpha}
\end{equation}
Note that we have $g_{d,e}=\frac{1}{d+1}\alpha_{-2,d-2,e}$.
Further, we define $a_{d,e}$ to be the minor determinant
$$a_{d,e} = \det (\alpha _{i,j,e}\ \big|\ 0 \leq i \leq e-1,\  d-1 \leq j \leq d+e-2).$$
Again, these polynomials, while presently unmotivated, arise naturally (see Lemma \ref{lem:vij} below).  We now state a conjecture that we will show implies the Polydegree Conjecture.

\renametheorem{Polydegree Ideal Conjecture}
\begin{namedtheorem} [$PIC(d,e)$]In $\IQ[x_1,\ldots,x_e]$, 
 ${\rm rad}(g_{d,e},\ldots,g_{d+e-1,e})$ is a maximal ideal and 
$a_{d,e} \notin {\rm rad}(g_{d,e},\ldots,g_{d+e-2,e})$.  
\end{namedtheorem}

\subsection{Main results}\label{sec:theorems}
In this paper, we will prove three results which can be used to conclude cases of the Polydegree Conjecture.  The logical connection between the three is

\begin{equation*}
\text{Theorem \ref{thm:partialSpecialization}} \quad\Longrightarrow\quad \text{Theorem \ref{thm:partialSpecializationSomeC}} \quad\Longrightarrow\quad \text{Theorem \ref{thm:ICDC}.}
\end{equation*}

\begin{theorem}\label{thm:ICDC}
If $PIC(d,e)$ holds, then $\mathcal{G}_{(d+e)} \subset \overline{\mathcal{G}_{(d,e+1)}} $.
\end{theorem}

We will prove Theorem \ref{thm:partialSpecialization} in Section \ref{sec:proof}, and the above implications here. Note that Theorem \ref{thm:ICDC} is immediately implied by
\begin{theorem}\label{thm:partialSpecializationSomeC}
Let $d\geq 2,e \geq 1$, and suppose that there exists a specialization homomorphism $\psi: \IC[x_1\ldots,x_e] \rightarrow \IC$ such that 
\begin{enumerate}
\item $\psi (g_{d+i,e}) =0$ for each $0 \leq i \leq e-2$;
\item $\psi (g_{d+e-1,e})\neq 0$; and
\item $\psi (a_{d,e}) \neq 0$.  
\end{enumerate}
Then $\mathcal{G}_{(d+e)} \subset \overline{\mathcal{G}_{(d,e+1)}}$.
\end{theorem}

By strengthening the hypotheses of this theorem, we can obtain a slightly stronger conclusion which will be necessary to answer a question of Arzhantsev in Section \ref{sec:Arzhantsev}.  To state this, we introduce from \cite{ZAK} the subgroups 
\begin{align*}
H_d &= \left\{ (X,Y+aX^d)\ \mid\ a \in \IC\right\}, &
K_d &= \left\{ (X+aY^d,Y)\ \mid\ a \in \IC\right\}.
\end{align*}

\begin{theorem}\label{thm:partialSpecialization}
Let $d\geq 2,e \geq 1$, and let $c_0 \in \IC^*$.  Suppose that there exists a specialization homomorphism $\psi: \IC[x_1\ldots,x_e] \rightarrow \IC$ such that 
\begin{enumerate}
\item $\psi (g_{d+i,e}) =0$ for each $0 \leq i \leq e-2$;
\item $\psi (g_{d+e-1,e})=c_0$; and
\item $\psi (a_{d,e}) \neq 0$.  
\end{enumerate}
Then, for any $c_1,\ldots,c_{d+e} \in \IC$, we have
 $$(X+\sum _{r=0} ^{d+e} c_r Y^{d+e-r},Y) \in  \overline{\mathcal{G}_{(d,e+1)}} \cap  \overline{\langle K_d,H_1,K_{e+1} \rangle}.$$
\end{theorem}

\begin{proof}[Proof that Theorem \ref{thm:partialSpecialization} implies Theorem \ref{thm:partialSpecializationSomeC}]
Suppose Theorem \ref{thm:partialSpecialization} is true.  Let $c_0 \in \IC^*$ and $c_1,\ldots,c_{d+e} \in \IC$.  It suffices to show that $$\theta := (X+\sum _{r=0} ^{d+e} c_r Y^{d+e-r},Y) \in  \overline{\mathcal{G}_{(d,e+1)}}.$$
Suppose  $\psi : \IC[x_1,\ldots,x_e] \rightarrow \IC$ is a specialization homomorphism satisfying  the hypotheses of Theorem \ref{thm:partialSpecializationSomeC} that
\begin{align*}
\psi (g_{d+i,e}) &=0 \text{ for each $0 \leq i \leq e-2$}, & 
\psi (g_{d+e-1,e})& \neq 0, & 
\psi (a_{d,e}) &\neq 0.
\end{align*}
Then we define $\lambda, b_0 \in \IC^*$ and $b_1,\ldots,b_{d+e} \in \IC$ by
\begin{align*}
b_0&=\psi(g_{d+e-1,e}), & \lambda & = \left(\frac{c_{0}}{b_{0}}\right)^\frac{1}{d+e}, &  b_i &= \frac{c_i}{\lambda ^{d+e-i}} \text{ for $1 \leq i \leq d+e$.}  
\end{align*}
By Theorem \ref{thm:partialSpecialization}, we have
$$\theta _0 := (X+\sum _{r=0} ^{d+e} b_r Y^{d+e-r},Y) \in  \overline{\mathcal{G}_{(d,e+1)}}.$$
However, it is easy to check that letting $\delta = (X,\lambda Y)$, we have $\theta = \delta ^{-1} \theta _0 \delta \in \overline{\mathcal{G}_{(d,e+1)}}$.
\end{proof}

The proof of Theorem \ref{thm:partialSpecialization} is given in Section \ref{sec:proof}.  
Theorem \ref{thm:ICDC} is aesthetically nicer and well suited for checking individual cases with the aid of computer algebra (see Section~\ref{sec:computer}). On the other hand, we find Theorem \ref{thm:partialSpecializationSomeC} and Theorem \ref{thm:partialSpecialization} more convenient for proving the results in Section \ref{sec:applications}.

\subsection{The Valuation Criterion}

The key technical tool we use below is the Valuation Criterion (due to Furter \cite{Furter09}), which we state in a slightly weaker version here:
\begin{theorem}\label{valuativeCriterion}
Let $\theta \in \G_{(d_1,\ldots,d_r)}(\IC(Z))$.  If $\theta \in \G(\IC[Z])$, then modulo $Z$ we have $\bar{\theta} \in \overline{\G_{(d_1,\ldots,d_r)}(\IC)}$.
\end{theorem}

This provides us with a very useful algebraic way to check the topological condition of containment in the closure.  Given $\theta \in \G(\IC(Z))$, we frequently verify $\theta \in \G(\IC[Z])$ by appealing to the overring principle (\cite{ArnoBook}, Lemma 1.1.8); if the Jacobian determinant of $\theta$ lies in $\IC^*$, to check $\theta \in \G(\IC[Z])$ one simply needs to verify that each component of $\theta$ lies in $\IC[Z][X,Y]$.

\begin{example}\label{ex:Nagata}
Let $\theta = (X+\frac{Y^2}{Z},Y) \circ (X,Y+Z^2X) \circ (X-\frac{Y^2}{Z},Y) \in \G_{(2,2)}(\IC(Z))$.  Simplifying, we see $$\theta = \left(X+2Y(XZ-Y^2)+Z(XZ-Y^2)^2,Y+Z(XZ-Y^2)\right). $$  
Since $\theta \in \G(\IC(Z))$ and both components lie in $\IC[Z]$, the overring principle implies that $\theta \in \G(\IC[Z])$.  Thus, applying the Valuation Criterion, we can conclude $\overline{\theta} = (X-2Y^3,Y) \in \overline{\G_{(2,2)}}$.

\end{example}

The previous example can easily be modified to show that $\G_{(3)} \subset \overline{\G_{(2,2)}}$.  In fact, the heart of our argument is a generalization of this construction (cf. \eqref{eq:tausigma} in the proof of Theorem \ref{thm:specialization}).  

\section{Proof of Theorem \ref{thm:partialSpecialization}}\label{sec:proof}
Throughout this section, let $d\geq 2$ and $e \geq 1$ be fixed.  In order to prove Theorem \ref{thm:partialSpecialization}, we will require an intermediate result (Theorem \ref{thm:specialization}); stating this will require some notation which we presently introduce.  We then prove Theorem \ref{thm:specialization} before using it to prove Theorem \ref{thm:partialSpecialization} in Section \ref{sec:t5}.  Before proceeding, the reader may wish to lightly read the proof of Theorem \ref{thm:specialization} first (located in Section \ref{sec:intermediate}) to understand the motivation of these definitions.

Define $R_{-1}=\IC$ and for each $0 \leq i \leq d+e$, set $R_i = R_{i-1}[u_{i,0},\ldots,u_{i,e-1}]$ for variables $u_{i,j}$ (so $R_i \cong \IC^{[(i+1)e]}$).   Further define $U_i(Y,Z) \in R_i[Y,Z]$ and $U(Y,Z) \in R_{d+e}[Y,Z]$ by 
\begin{align*}
U_i (Y,Z) &= \sum _{j=0} ^{e-1} u_{i,j} Y^{j+2}Z^j, &
U(Y,Z) &= \sum _{i=0} ^{d+e} U_i(Y,Z) Z^i.
\end{align*}
We will use $\deg _{(1,-1)}$ to denote the $(1,-1)$-degree grading on $R_{d+e}[Y,Z]$; that is, $\deg _{(1,-1)} Y = 1$ and $\deg _{(1,-1)} Z=-1$.  It is immediate from the preceding definitions that $U(Y,Z) \in (Y^2)$.  Moreover, each $U_i$ is homogeneous of $\deg _{(1,-1)} U_i(Y,Z) = 2$, and thus $\deg _{(1,-1)} U(Y,Z)=2$. Let $I(Y,Z) \in R_{d+e}[Z][[Y]]$ be the formal inverse of $Y+ZU(Y,Z)$, so that
\begin{equation}\label{eq:YI}
Y=I(Y+ZU(Y,Z),Z). 
\end{equation}

We make a few observations about these definitions before proceeding.
\begin{lemma} With $U(Y,Z)$ and $I(Y,Z)$ defined as above, we have
\begin{enumerate}
\item $U(I(Y,Z),Z) \in (Y^2)$,
\item $\deg _{(1,-1)} I(Y,Z) = 1$,
\item $ \deg _{(1,-1)} U(I(Y,Z),Z) = 2$.
\end{enumerate}
\end{lemma}
\begin{proof}
First, note that, since $U \in (Y^2)$, we have $I \in (Y)$, and thus $U(I(Y,Z),Z) \in (Y^2)$.  
For the second statement, since $\deg _{(1,-1)} U=2$, we see $\deg _{(1,-1)} \left( Y+ZU(Y,Z) \right)= 1$, and thus, from \eqref{eq:YI}, $\deg _{(1,-1)} I(Y,Z) = \deg _{(1,-1)} I(Y+ZU(Y,Z),Z) = \deg _{(1,-1)} Y = 1$.  The third statement follows immediately from the second since $\deg _{(1,-1)} U(Y,Z)=2$.
\end{proof}
It follows from the third part of the previous lemma that we can write 
\begin{align}
U(I(Y,Z),Z)&=\sum _{i=0} ^\infty V_i(Y,Z)Z^i & \text{where} & & V_i(Y,Z)&=\sum _{j=0} ^\infty v_{i,j} Y^{j+2}Z^j \label{eq:Vi}
\end{align}
for some $v_{i,j} \in R_{d+e}$.  Note each $V_i$ is homogeneous of $\deg _{(1,-1)} V_i = 2$.  

\subsection{An intermediate theorem}\label{sec:intermediate}
We are now ready to state the intermediate specialization theorem.
\begin{theorem}\label{thm:specialization}Let $d\geq 2,e \geq 1$, let $c_0 \in \IC^*$, and let $c_1,\ldots,c_{d+e} \in \IC$.  
Suppose that there exists a specialization homomorphism $\psi : R_{d+e} \rightarrow \IC$ such that, for each $0 \leq r \leq e-1$, 
\begin{enumerate}
\item $\psi (v_{r,d-1 +i }) = 0$ for each $0 \leq i <e-1-r$; and
\item $\psi (v_{r,d+e-2-r}) = c_{r}$.
\end{enumerate}
Then
 $$(X+\sum _{r=0} ^{d+e} c_r Y^{d+e-r},Y) \in  \overline{\mathcal{G}_{(d,e+1)}}.$$
 Moreover, if $\psi(u_{0,e-1}) \neq 0$, then $(X+\sum _{r=0} ^{d+e} c_r Y^{d+e-r},Y)\in  \overline{\langle K_f,H_1,K_{e+1} \rangle}$  for some $2 \leq f \leq d$ as well.
\end{theorem}
\begin{proof}
Let $\psi: R_{d+e} \rightarrow \IC$ be a specialization homomorphism of the hypothesized form; for notational convenience, for the remainder of the proof we will identify $u_{i,j}$ and $v_{i,j}$ with their images under this specialization homomorphism (and write, for instance, $v_{0,d+e-2}=c_0$).
We begin by defining
\begin{align*}
T(Y) = \sum _{r=0} ^{d+e} c_r Y^{d+e-r}& &\text{and}& & \theta = (X+T(Y),Y).  
\end{align*}
The key tool we use is Theorem \ref{valuativeCriterion}.  We construct $\sigma \in \G(\IC(Z))$ such that $\sigma$ has polydegree $(f,g)$ for some $f \leq d$ and $g \leq e+1$, and (when $u_{0,e-1} \neq 0$) is a product of elements of $K_f$,$H_1$, and $K_{e+1}$ (with coefficients in $\IC(Z)$ rather than $\IC$).   Moreover, we will show that the components of $\sigma$ lie in $\IC[Z][X,Y]$; then, letting $\bar{\sigma}$ denote the image of $\sigma$ modulo $Z$, we will show $\bar{\sigma}=\theta$, which by Theorem \ref{valuativeCriterion} implies $\theta \in  \overline{\mathcal{G}_{(f,g)}} \subset \overline{\G_{(d,e+1)}}$ (with the last containment following from a result of Furter \cite{Furter09}), and in the case $u_{0,e-1} \neq 0$, that $\theta \in  \overline{\langle K_f,H_1,K_{e+1} \rangle}$ as well.

We define $V(Y,Z) \in \IC[Y,Z]$ to be a truncation of (the specialization of) the power series $U(I(Y,Z),Z)$, namely 
\begin{equation} \label{eq:V}
V(Y,Z) =  \sum _{i=0} ^{d+e-2} \sum _{j=0} ^{d-2} v_{i,j} Y^{j+2}Z^{i+j}.
\end{equation}

Note that $\deg _{(1,-1)} V(Y,Z) \leq 2$ and $\deg _Y V(Y,Z) \leq d$.
Set $T_0(Y) = \sum _{r=e} ^{d+e} c_r Y^{d+e-r}$, and define $\tau _1, \tau _2 \in \B(\IC(Z))$ and $\alpha \in \A\left(\IC[Z])\right)$ by 
\begin{align}
\tau _1 &= \left( X-\frac{V(Y,Z)}{Z^{d+e-2}}+T_0(Y),Y\right), & \alpha &= \left(X,Y+Z^{d+e-1}X\right), & \tau _2 &= \left(X+\frac{U(Y,Z)}{Z^{d+e-2}},Y \right). \label{eq:tausigma}
\end{align}

Let $$f=\deg_Y \left( \frac{-V(Y,Z)}{Z^{d+e-2}}+T_0(Y)\right)$$ and $g=\deg _Y U(Y,Z)$. Note that $f \leq d$ by construction, and unless $V(Y,Z)=0$, then $2 \leq f$ (we will see momentarily that $V(Y,Z) \neq 0$).   Moreover, if $u_{0,e-1} \neq 0$ then $g=e+1$.

Clearly $\sigma := \tau _1 \alpha \tau _2 \in \G_{(f,g)}\left(\IC(Z)\right)$.  
A straightforward computation gives, for some $\tilde{T} \in \IC[Z][X,Y],$
\begin{equation}
\sigma = \left(X+\frac{U(Y,Z)-V(Y+ZU,Z)}{Z^{d+e-2}}+T_0(Y)+Z\tilde{T}, Y+ZU(Y,Z)+Z^{d+e-1}X \right).\label{eq:sigma}
\end{equation}

Let us define $W(Y,Z)=U(Y,Z)-V(Y+ZU,Z)$.  If $W \in (Z^{d+e-2})$, then both components of $\sigma$ lie in $\IC[Z][X,Y]$.  Then, letting $\bar{\sigma}$ be the image of $\sigma$ after going modulo $Z$, showing $\bar{\sigma}=\theta$ will complete the proof.  Note that showing $W \in (Z^{d+e-2})$ will also give the above assertion that $V(Y,Z) \neq 0$.

Thus our remaining task is to show  $W \in (Z^{d+e-2})$ and compute $W$ modulo $Z^{d+e-1}$.   
Set $W_0(Y,Z)=U(I(Y,Z),Z)-V(Y,Z)$.  Then from \eqref{eq:V}, together with the hypothesis on the specialization homomorphism that $v_{i,j}=0$ when $i+j<d+e-2$, we see that, for some $W_1(Y,Z) \in R_{d+e}[Y,Z]$,
\begin{align*}
W_0(Y,Z)&= Z^{d+e-2} \sum _{i=0} ^{e-1}  v_{i,d+e-2-i} Y^{d+e-i} +Z^{d+e-3}W_1(Y,Z) \\
&= Z^{d+e-2} \sum _{i=0} ^{e-1}  c_i Y^{d+e-i} +Z^{d+e-3}W_1(Y,Z).
\end{align*}
Now, we observe
$$W(Y,Z)=U(Y,Z)-V(Y+ZU,Z)=U(I(Y+ZU,Z),Z)-V(Y+ZU,Z)=W_0(Y+ZU,Z).$$
Then clearly we have $W(Y,Z) \in (Z^{d+e-2})$ and 
$$W(Y,Z)+Z^{d+e-2}T_0(Y) \equiv \sum _{r=0}^{d+e} c_{r}Y^{d+e-r} Z^{d+e-2} \pmod{Z^{d+e-1}}.$$
Combining this with \eqref{eq:sigma}, we have $\overline{\sigma}=(X+T(Y),Y)=\theta$ as required.
 \end{proof}

\subsection{Proof of Theorem \ref{thm:partialSpecialization}}\label{sec:t5}
Next, we use Theorem \ref{thm:specialization} to prove Theorem \ref{thm:partialSpecialization}.  The bulk of the work is done in Lemma \ref{lem:vij} below.  The following theorem is well known in various forms; a  similar use appears in \cite{EdoArno}.

\begin{theorem}\cite[(4.5.12)]{morseFeshbach} \label{thm:LagrangeInversion}
Let $R$ be a $\IQ$-algebra, and let \[f(Y) = Y+\sum _{j=1} ^\infty f_j Y^{j+1} \in R[[Y]]\] have formal (compositional) inverse \[h(Y)=Y+\sum _{j=1} ^\infty h_j Y^{j+1} \in R[[Y]].\]  
Then 
\[h_j = \frac{1}{j+1} \sum _{{\bf a} \cdot \IN = j} (-1)^{|{\bf a}|} \binom{|{\bf a}|+j}{{\bf a}, j} f_1 ^{a_1} \cdots f_{j} ^{a_j}.\]
\end{theorem}

\begin{lemma} \label{lem:vij}
Let  $j \geq 0$.  Then
\begin{enumerate}
\item \[ v_{0,j}= -\frac{1}{j+2} \sum _{ {\bf a} \cdot \IN = j+1} (-1)^{|{\bf a}|} \binom{|{\bf a}|+j+1}{{\bf a}, j+1} u_{0,0} ^{a_1} \cdots u_{0,e-1} ^{a_{e-1}}. \]
\item If $i \geq 1$, there exist $p_{i,j} \in R_{i-1}$ such that
\[v_{i,j} = p_{i,j} + \sum _{s=0} ^{e-1} \alpha _{s,j,e} u_{i,s}, \]
where $\alpha _{s,j,e}$ is as defined in \eqref{eq:alpha}.
\end{enumerate}
\end{lemma}

\begin{proof}
We apply Theorem \ref{thm:LagrangeInversion} to $Y+ZU(Y) \in R_{d+e}[Z][Y]$; recall that we defined $I(Y) \in R_{d+e}[Z][[Y]]$ to be its formal inverse.  
\begin{equation*}
  Y + Z U (Y) = Y +  \sum_{i = 0}^{d + e} \sum_{j = 0}^{e - 1} u_{i,
   j} Y^{j+2} Z^{i + j+1} = Y + \sum_{k = 1}^{e} Y^{k+1} \left(\sum_{i =
   0}^{d + e} u_{i, k-1} Z^{i + k} \right) .
\end{equation*}
  
Writing $I (Y) = Y + \sum _{j=1} ^\infty h_j Y^{j+1}$ for some $h_j \in R_{d+e}[Z]$, Theorem \ref{thm:LagrangeInversion} implies
\begin{equation*}
  h_m = \frac{1}{m + 1} \sum_{\boldsymbol{a} \cdot \mathbb{N}= m} (- 1)^{|
  \boldsymbol{a} |} \binom{| \boldsymbol{a} | + m}{\boldsymbol{a}, m} f_1^{a_1} \cdots f_m ^{a_m},
\end{equation*}
where $f_n \in R_{d+e}[Z]$ is the coefficient of $Y^{n+1}$ in $ZU(Y)$, namely $f_n =   \sum_{i = 0}^{d + e} u_{i, n - 1} Z^{i+n}.$  Factoring out the $Z^n$ from each $f_n$, we can write $F_n = \sum_{i = 0}^{d + e} u_{i, n - 1} Z^{i}$, and then we have 
\begin{equation}
  h_m = \frac{1}{m + 1}Z^m \sum_{\boldsymbol{a} \cdot \mathbb{N}= m} (- 1)^{|
  \boldsymbol{a} |} \binom{| \boldsymbol{a} | + m}{\boldsymbol{a}, m} F_1^{a_1} \cdots F_m ^{a_m}.
   \label{eq:hm}
\end{equation}
Since $I (Y) + Z U (I (Y)) = Y$, we have
\begin{equation}\label{eq:Uh}
  U (I (Y)) = Z^{-1}\left(Y - I (Y)\right) = -\sum _{j=1} ^\infty h_j Y^{j+1} Z^{-1}.
\end{equation}

Now, for the first part, recall from \eqref{eq:Vi} that $v_{0,j}$ is the coefficient of $Y^{j+2}Z^j$ in $U(I(Y))$.  Then from \eqref{eq:Uh}, we see that $v_{0,j}$ is the coefficient of $Z^{j+1}$ in $-h_{j+1}$, so we deduce from \eqref{eq:hm} that
$$v_{0,j} = -\frac{1}{j+2} \sum _{{\bf a} \cdot \IN = j+1} (-1)^{|{\bf a}|} \binom{|{\bf a}|+j+1}{{\bf a}, j+1} u_{0,0}^{a_1} \cdots u_{0,e-1} ^{a_e}.$$

For the second part, we assume $i \geq 1$.  We proceed similarly to the first part but, to simplify some computations, we consider the natural degree grading on $R_{i}=R_{i-1}[u_{i,0},\dots,u_{i,e-1}]$; that is, elements of $R_{i-1}$ have degree zero, while the variables $u_{i,j}$ have degree 1.   We will compute $\overline{v_{i,j}}$, the leading term of $v_{i,j}$ in this grading.

  Recalling from \eqref{eq:Vi} that $v_{i,j}$ is the coefficient of $Y^{j+2}Z^{i+j}$ in $U(I(Y))$, we  see from \eqref{eq:Uh} that $v_{i,j}$ is the coefficient of $Z^{i+j+1}$ in $-h_{j+1}$, and thus, from \eqref{eq:hm},
\begin{align*}
\overline{v_{i,j}} &= -\frac{1}{j+2} \sum _{s=1} ^e u_{i,s-1} \sum _{\substack{{\bf a} \cdot \IN = j+1 \\ a_s > 0}} (-1)^{|{\bf a}|} \binom{|{\bf a}|+j+1}{{\bf a},j+1} \frac{u_{0,0}^{a_1}\cdots u_{0,e-1}^{a_e}}{u_{0,s-1}}. 
\end{align*}

  Since
  \begin{equation*}
    a_s \binom{| \boldsymbol{a} | + j + 1}{\boldsymbol{a}, j + 1} = (j + 2)
     \binom{| \tilde{\boldsymbol{a}} | + j + 2}{\tilde{\boldsymbol{a}}, j + 2},
  \end{equation*}
  where $\tilde{\boldsymbol{a}}$ is obtained from $\boldsymbol{a}$ by replacing
  $a_s$ with $a_s - 1$, we find
  \begin{eqnarray*}
    \overline{v_{i, j}} & = &  \sum_{s = 1}^e u_{i, s - 1} \sum_{\boldsymbol{a}
    \cdot \mathbb{N}= j + 1 - s} (- 1)^{| \boldsymbol{a} |} \binom{|
    \boldsymbol{a} | + j + 2}{\boldsymbol{a}, j + 2} u_{0, 0}^{a_1} u_{0, 1}^{a_2}
    \cdots u_{0, e - 1}^{a_e}\\
    & = &  \sum_{s = 0}^{e - 1} u_{i, s} \sum_{\boldsymbol{a} \cdot
    \mathbb{N}= j - s} (- 1)^{| \boldsymbol{a} |} \binom{| \boldsymbol{a} | + j +
    2}{\boldsymbol{a}, j + 2} u_{0, 0}^{a_1} u_{0, 1}^{a_2} \cdots u_{0, e -
    1}^{a_e}\\
    & = &  \sum_{s = 0}^{e - 1} u_{i, s} \alpha_{s, j,e}.
  \end{eqnarray*}
 Thus we see $v_{i,j}$ has degree one, so $v_{i,j} = p_{i,j} +  \sum_{s = 0}^{e - 1} u_{i, s} \alpha_{s, j,e}$ for some $p_{i,j}$ of degree zero, i.e. $p_{i,j} \in  R_{i-1}$,  as claimed.

\end{proof}

\begin{remark}\label{remark:vg}
We note that the first part of Lemma \ref{lem:vij} is stating that $v_{0,k}$ and $g_{k+1,e}$ are the same polynomials after a change of variables.
\end{remark}

\begin{proof}[Proof of Theorem \ref{thm:partialSpecialization}]
Let $c_0 \in \IC^*$ and $c_1,\ldots,c_{d+e} \in \IC$.  Note that, by Remark \ref{remark:vg}, the hypotheses are that there exist a specialization homomorphism $\psi _0: R_0 \rightarrow \IC$ such that 
\begin{enumerate}
\item $\psi _0 (v_{0,d-1+i}) =0$ for each $0 \leq i \leq e-2$;
\item $\psi _0 (v_{0,d+e-2})=c_0$; and
\item $\psi _0 (a_{d,e}) \neq 0$.  
\end{enumerate}

From Lemma \ref{lem:vij}, since $\psi _0 (a_{d,e}) \neq 0$, we have for each $r \geq 1$ that 
$$R_r=R_{r-1}[v_{r,d-1},\ldots,v_{r,d+e-2}].$$

Starting from $\psi _0$, we can thus inductively construct a $\IC$-algebra map $\psi _r : R_{r} \rightarrow \IC$ for each $1 \leq r \leq d+e$ satisfying $\psi _r |_{R_{r-1}} = \psi _{r-1}$,
\begin{align*}
\psi _r (v_{r,d-1 +i }) = 0 \quad\text{for each $0 \leq i <e-1-r$}
\end{align*}
and $\psi _r (v_{r,d+e-2-r}) = c_{r}$.
The map $\psi _{d+e}$ satisfies the hypotheses of Theorem \ref{thm:specialization}, completing the proof of Theorem \ref{thm:partialSpecialization}.
\end{proof}

\section{Applications to the Polydegree Conjecture}\label{sec:applications}
In this section, we show how our main theorem implies various results related to the Polydegree Conjecture.  In particular, we first show how we very quickly obtain Edo's result (Theorem \ref{thm:edo}), and then also give a short proof of Furter's result (Theorem \ref{thm:furter}).  Finally, we address a question of Arzhantsev that arises naturally in \cite{ZAK}.  

\subsection{Computer-aided results}\label{sec:computer}
We note that, by virtue of Theorem \ref{thm:ICDC}, it is easy to establish
\begin{theorem}
The Polydegree Conjecture is algorithmically decidable for fixed $d$ and $e$.
\end{theorem}
Indeed, since the polynomials $g_{d,e}$ are of positive order, if the ideal ${\rm rad}(g_{d,e}, \ldots, g_{d+e-1,e}$) is maximal it must be the ideal $(x_1,\ldots,x_e)$.  Thus $PIC(d,e)$ is an ideal membership question, giving the theorem above.

By computing Gr\"obner bases, we can verify whether $PIC(d,e)$ holds for particular $d$ and $e$.  The following lemma is useful in this regard:
\begin{lemma}\label{lem:computation}
Let $d,e \geq 2$.  Suppose $PIC(d,e-1)$ is true, and
\begin{enumerate}
\item\label{i:lem:computation:1} $x_e \in {\rm rad}(g_{d,e},\ldots,g_{d+e-1,e})$,
\item $x_e \in {\rm rad}(g_{d,e},\ldots,g_{d+e-2,e},a_{d,e})$.
\end{enumerate}
Then $PIC(d,e)$ is true as well.
\end{lemma}
\begin{proof}
The key observation is that $g_{d,e} \equiv g_{d,e-1} \pmod {x_e}$.  Since $PIC(d,e-1)$ holds,
$${\rm rad}(g_{d,e-1},\ldots,g_{d+e-2,e-1},x_e)=(x_1,\ldots,x_e),$$
and hypothesis \ref{i:lem:computation:1} implies $$ {\rm rad}(g_{d,e},\ldots,g_{d+e-1,e}) = {\rm rad}(g_{d,e},\ldots,g_{d+e-1,e},x_e) = {\rm rad}(g_{d,e-1},\ldots,g_{d+e-1,e-1},x_e)= (x_1,\ldots,x_e).$$  Similarly,
\begin{align*}
{\rm rad}(g_{d,e},\ldots,g_{d+e-2,e},a_{d,e}) &\supset {\rm rad}(g_{d,e},\ldots,g_{d+e-2,e},x_e) \\
&= {\rm rad}(g_{d,e-1},\ldots,g_{d+e-2,e-1},x_e) = (x_1,\ldots,x_e).
\end{align*}
Thus ${\rm rad}(g_{d,e},\ldots,g_{d+e-2,e},a_{d,e})$ is maximal (in $\IQ^{[e]}$), and since  ${\rm rad}(g_{d,e},\ldots,g_{d+e-2,e})$ is not, we see $a_{d,e} \notin {\rm rad}(g_{d,e},\ldots,g_{d+e-2,e})$.
\end{proof}

We used Sage combined with Magma (and, independently, Mathematica; though not in all cases) to verify the two hypotheses of Lemma \ref{lem:computation} and establish the following results:
\begin{theorem}\mbox{}
\begin{enumerate}
\item If $2 \leq d < 50$, then $\G_{(d+3)} \subset \overline{G_{(d,4)}}$.
\item If $2 \leq d < 20$, then $\G_{(d+4)} \subset \overline{G_{(d,5)}}$.
\item If $2 \leq d \leq 12$, then  $\G_{(d+5)} \subset \overline{G_{(d,6)}}$.
\end{enumerate}
\end{theorem}

We note that these include the previously unknown special cases $\G_{(8)} \subset \overline{\G_{(5,4)}}$, $\G_{(10)} \subset \overline{\G_{(6,5)}}$, and $\G_{(12)} \subset \overline{\G_{(7,6)}}$ (among others).  These would be quite tedious to establish directly (with currently known techniques), as the Gr\"obner bases involved in the computations can consist of a large number of polynomials, depending on the term order chosen.

\subsection{Edo's Theorem}\label{sec:edo}
\begin{theorem}[Edo]\label{thm:edo} If $(d-1)=me$ for some $m \in \IN$ and $d,e \geq 2$, then $\mathcal{G}_{(d+e)} \subset \overline{\mathcal{G}_{(d,e+1)}}$.
\end{theorem}
\begin{proof}
Rather than use Theorem \ref{thm:ICDC}, we find it easier to directly construct the specialization homomorphisms.
Let $c_0,\ldots,c_{d+e} \in \IC$ with $c_0 \neq 0$.  
We make use of Theorem \ref{thm:partialSpecialization} and define a specialization homomorphism $\psi : \IC[x_1,\ldots,x_e] \rightarrow \IC$ by $$\psi(x_j) = \begin{cases} 0, & \text{if $0 \leq j <e$,} \\ \left( \frac{ (-1)^{m+1} c_0}{\binom{d+e+m}{m+1}} \right)^{\frac{1}{m+1}}, & \text{if $j=e$.}\end{cases}$$

Then $\psi(g_{d,e})=\cdots=\psi(g_{d+e-2,e})=0$, while $\psi(g_{d+e-1,e})=c_0$.  Note that these choices mean $\psi(\alpha _{i,j,e})=0$ whenever $j-i$ is not a multiple of $e$.  Thus, to see $\psi(a_{d,e}) \neq 0$, it suffices to show that $\psi(\alpha _{i,i+ne,e}) \neq 0$ for all $n \in \IN$. But this is immediate from the definition \eqref{eq:alpha}.
\end{proof}

\begin{remark}
Edo's original proof involved showing that certain hypergeometric polynomials did not share a common root.  Interestingly, our approach does not require this to prove Theorem \ref{thm:edo}; however, we will use these same hypergeometric polynomials below to prove Furter's Theorem \ref{thm:furter}.
\end{remark}

\subsection{Furter's Theorem}\label{sec:furter}
We now turn our attention to providing a new proof of Furter's result that the Polydegree Conjecture holds for $e=2$.  We will again make use of Theorem \ref{thm:partialSpecialization}; however, in contrast to the previous proof, it is more difficult to verify that $\psi (a_{d,2}) \neq 0$.  Thus, we first develop relations among the polynomials $g_{k,2}$ and $\alpha _{i,j,e}$.  To do so, we define polynomials $P_{k,n}(z) \in \IQ[z]$ (for all positive integers $k,n$) as
$$P_{k,n}(z) = \sum _{b=0} ^{\lfloor \frac{n}{2} \rfloor } \binom{ k+n-b}{ k, b,n-2b} z^b.$$

These are actually re-indexed versions of polynomials introduced in \cite{SFPA2}, where it was observed that they are actually hypergeometric polynomials, namely (in our indexing)
\begin{equation} \label{eq:hypergeometric}
P_{k,n}(z)=\binom{k+n}{k} \;{}_2F_1\left( -\frac{n}{2}, -\frac{n}{2}+\frac{1}{2}; -k-n; -4z\right).  
\end{equation}

\begin{lemma} \label{lem:P}
Let $i,j,k \in \IN$ with $j>i$. Then
\begin{enumerate}
\item $\alpha _{i,j,2} = (-x_1)^{j-i} P_{j+2,j-i}\left(-\frac{x_2}{x_1^2} \right)$,
\item $g_{k,2} =  \frac{ (-x_1)^k}{k+1} P_{k,k}\left( -\frac{x_2}{x_1^2}\right)$,
\item $a_{k,2} =(-x_1)^{2k-2}\left( P_{k+1,k-1}\left(-\frac{x_2}{x_1^2} \right)P_{k+2,k-1}\left(-\frac{x_2}{x_1^2} \right)-P_{k+1,k-2}\left(-\frac{x_2}{x_1^2} \right)P_{k+2,k}\left(-\frac{x_2}{x_1^2} \right) \right).$
\end{enumerate}
\end{lemma}
\begin{proof}
This is a straightforward computation; first, for $\alpha _{i,j,2}$ we observe 
\begin{align*}
\alpha _{i,j,2} &= \sum _{a_1+2a_2 = j-i} (-1)^{|{\bf a}|} \binom{2a_1+3a_2+i+2}{a_1, a_2,j+2} x_1 ^{a_1}\  x_2 ^{a_2} \\
&= \sum _{b=0} ^{\lfloor \frac{j-i}{2} \rfloor} (-1) ^{j-i-b} \binom{2j-i+2-b}{ b, j+2,j-i-2b} x_1 ^{j-i-2b} x_2 ^b.
\end{align*}
Factoring $(-x_1)^{j-i}$ out of the last sum, we find the claimed formula for $\alpha _{i,j,2}$.
 
The second claim then follows from the observation that $g_{k,2} = \frac{1}{k+1}\alpha_{k-2,-2,2}$.
The third claim follows immediately from the definition $a_{d,2}=\alpha _{0,d-1,2} \alpha _{1,d,2}-\alpha _{1,d-1,2} \alpha _{0,d,2} $ and the first claim.
\end{proof}

Now, we state a theorem about the polynomials $P_{k,n}$ that will immediately imply Furter's result (restated below in Theorem \ref{thm:furter}).   
\begin{theorem}\label{thm:lambda}
Let $d \geq 2$.  Then there exists $\lambda \in \IC^*$ such that
\begin{enumerate}
\item $P_{d,d}(\lambda) = 0$,
\item $P_{d+1,d+1}(\lambda) \neq 0$,
\item $P_{d+1,d-1}(\lambda)P_{d+2,d-1}(\lambda)-P_{d+1,d-2}(\lambda)P_{d+2,d}(\lambda) \neq 0$.
\end{enumerate}
\end{theorem}

\begin{theorem}[Furter] \label{thm:furter} If $d \geq 2$, then $\mathcal{G}_{(d+2)} \subset \overline{\mathcal{G}_{(d,3)}}$.
\end{theorem}
\begin{proof}
We construct the specialization homomorphism $\psi _0: \IC[x_1,x_2] \rightarrow \IC$ in two steps.  First, let $\lambda \in \IC^*$ be as in Theorem \ref{thm:lambda}, and define $\psi _1: \IC[x_1,x_2] \rightarrow \IC[x_1]$ by $\psi_1(x_2)=\lambda x_1^2$.  Then Lemma \ref{lem:P} provides 
\begin{align*}
\psi _1 (g_{d,2}) &= 0, & \psi _1 (g_{d+1,2}) &\neq 0, & \psi _1 (a_{d,2}) &\neq 0.
\end{align*}
Now we observe that $\psi _1(g_{d+1,2})$ is a homogeneous polynomial of degree $d+1$ in a single variable; thus, for any $c_0 \in \IC^*$, there exists $\mu \in \IC^*$ such that, defining $\psi_2: \IC[x_1] \rightarrow \IC$ by $\psi_2(x_1)=\mu$, $\psi_2(\psi_1(g_{d+1,2}))=c_0$.  Similarly, $\psi _1(a_{d,2})$ is also homogeneous, so $\psi_2(\psi_1(a_{d,2})) \in \IC^*$.  The theorem then follows from Theorem \ref{thm:partialSpecialization} by setting $\psi_0=\psi _2 \circ \psi_1$.
\end{proof}

Now we return our attention to Theorem \ref{thm:lambda}.  For the sake of brevity, we will write simply $P_{k,n}=P_{k,n}(z)$; our present goal is to relate the determinant $D := P_{d+1,d-1}P_{d+2,d-1}-P_{d+1,d-2}P_{d+2,d}$ to the polynomials $P_{d,d}$ and $P_{d+1,d+1}$.  Recall from \eqref{eq:hypergeometric} that each $P_{k,n}$ is a hypergeometric polynomial; consequently, any polynomial $P_{d + m, d + n}$, where $m, n$ are integers,
can be expressed as a linear combination in $P_{d, d}$ and $P_{d + 1, d + 1}$,
with coefficients that are rational functions in $d$ and $z$. For instance, we
have
\begin{equation*}
  P_{d + 1, d - 1} = \frac{P_{d + 1, d + 1}}{1 + 3 z} - \frac{(3 d + 2) P_{d,
   d}}{(d + 1) (1 + 3 z)} .
\end{equation*}
Expressing each term of $D$ in terms of $P_{d, d}$ and $P_{d + 1, d + 1}$, it is straightforward
(but somewhat tedious) to obtain
\begin{eqnarray*}
  D & = & \frac{(d + 1) P_{d + 1, d + 1}^2}{2 (d + 2) z (1 + 3 z)} - \frac{3 d
  (3 d + 2) (3 d + 4) z P_{d, d}^2}{2 (d + 1)^2 (d + 2) (1 + 3 z) (1 + 4 z)}\\
  &  & - \frac{(2 + 8 z + d (2 d + 3) (2 + 9 z)) P_{d, d} P_{d + 1, d + 1}}{2
  (d + 1) (d + 2) z (1 + 3 z) (1 + 4 z)} .
\end{eqnarray*}
Thus, the theorem follows as long as there is a $\lambda \in
\mathbb{C}\backslash \{ 0, - 1 / 3, - 1 / 4 \}$ with $P_{d, d} (\lambda) = 0$
and $P_{d + 1, d + 1} (\lambda) \neq 0$. This final assertion is proved in the next lemma.

\begin{lemma}
  Let $\lambda \in \mathbb{C}$ be such that $P_{d, d} (\lambda) = 0$.  Then $P_{d +
  1, d + 1} (\lambda) \neq 0$; moreover, $\lambda \notin \{0,-\frac{1}{3}, -\frac{1}{4}\}$.
\end{lemma}

\begin{proof}
  For brevity, we will write $p_d (\lambda) = P_{d, d} (\lambda)$.  We first observe two (easily verified) relations that will prove useful:
  \begin{align}
    2 z (3 z + 1) p_d' (z)& = (3 d z + 4 d + 2) p_d (z) - (d + 1) p_{d + 1}
     (z) \label{eq:relation2} \\
    z (4 z + 1) p_d'' (z) &= 2 ((2 d - 3) z + d) p_d' (z) - d (d - 1) p_d (z) .\label{eq:relation3}
  \end{align}
  
  Clearly,
  $0$ is not a root of $p_d$ for any integer $d \geq 0$. We claim that,
  similarly, $- 1 / 3$ is not a root of any of these polynomials. Indeed, specializing \eqref{eq:relation2} to $z=-\frac{1}{3}$, we can deduce 
  the closed formula
  \begin{equation*}
    p_d (- 1 / 3) = 3^d \prod_{k = 1}^d \left(1 - \frac{1}{3 k} \right).
  \end{equation*}
Similarly, specializing at $z=-\frac{1}{4}$, one finds that
  \begin{equation*}
  p_d (- 1 / 4) = \left(\frac{3}{4} \right)^d \prod_{k = 1}^d \frac{9 k^2 -
   1}{k (2 k + 1)},
\end{equation*}
which implies that $- 1 / 4$ is not a root of $p_d$ for any $d$.

  Now suppose for contradiction that $\lambda$ is a root of both $p_d$ and $p_{d
  + 1}$. By the previous observations, $\lambda \neq 0, - 1 / 3$. It therefore
  follows from \eqref{eq:relation2}
  that $\lambda$ is also a root of $p_d'$. This, however, cannot be the case
  since, as observed in \cite{SFPA2}, the roots of $p_d$ are simple; indeed, the
  simplicity of the roots is an immediate consequence of the relation \eqref{eq:relation3}
  which implies that, if $p_d$ and $p_d'$ were to share a root, then all
  higher derivatives would share that root as well (thus implying that $p_d$
  vanishes identically).
\end{proof}

\subsection{A question of Arzhantsev on root subgroups}\label{sec:Arzhantsev}
In \cite{ZAK}, Arzhantsev, Kuyumzhiyan and Zaidenberg consider the question of when a subgroup of the automorphism group of a toric variety acts infinitely transitively on its open orbit.  In studying this question for the variety $\IC ^2$, it is natural to study the (Demazure) root subgroups $H_d, K_d \subset \mathcal{G}$ for each $d \in \IN$ defined by
\begin{align*}
H_d &= \left\{ (X,Y+aX^d)\ \mid\ a \in \IC\right\}, &
K_d &= \left\{ (X+aY^d,Y)\ \mid\ a \in \IC\right\}.
\end{align*}
In personal communication to the first author \cite{Arzhantsev}, Arzhantzev asked 
\begin{question}\label{q:arzhantsev}
For which $d \geq 2$ is $H_d \subset \overline{\langle H_2,K_1 \rangle}$?
\end{question}

Letting $(2)^d$ represent the length $d$ sequence $(2,\ldots,2)$, we have that $\langle H_2, K_1 \rangle \subset \coprod _{d \geq 0} \G_{(2)^d}$; moreover, $H_d \subset \G_{(d)}$ for all $d \geq 2$.  Noting that by Theorem \ref{thm:edo} we have $\G_{(d)} \subset \overline{\G_{(d-1,2)}}$, inducting on $d$ yields
\begin{theorem} If $d \geq 2$, then  $\G_{(d)} \subset \overline{\G_{(2)^{d-1}}}$.  In particular, $H_d \subset \overline{\G_{(2)^{d-1}}}$.
\end{theorem}

However, Question \ref{q:arzhantsev} is slightly stronger, as the subgroup $\langle H_2, K_1 \rangle$ does not contain the affine subgroup.  We know of two proofs \cite{SFPA2,Furter15} that $\G_{(d)} \subset \overline{\G_{(d-1,2)}}$, but both of these make use of the affine subgroup.  However, by using Theorem \ref{thm:partialSpecialization} to recover Edo's result in Section \ref{sec:edo}, we have actually proved something slightly stronger, namely that $\mathcal{G}_{(d+1)} \subset \overline {\langle K_d,H_1,K_2 \rangle}$ for all $d \geq 0$.  Since $K_d \subset \mathcal{G}_{(d)}$, we induct downwards on $d$ to obtain
\begin{theorem}
If $d \geq 2$,  then $K_d \subset \overline{\langle K_2,H_1 \rangle}$ and $H_d \subset \overline{\langle H_2,K_1 \rangle}$.
\end{theorem}
This implies (see \cite{ZAK}):
\begin{corollary}
The group $\langle H_2,K_1\rangle$ acts infinitely transitively on its open orbit.
\end{corollary}

\bibliography{bibliography}{}
\bibliographystyle{plain}

\end{document}